\documentclass[11pt]{amsart}

\usepackage{amsmath, amssymb, amsthm, amscd}
\usepackage{url}
\usepackage{graphicx}
\usepackage{color}
\usepackage{import}
\usepackage[margin=3cm]{geometry}
\usepackage{caption}
\usepackage{bbold}
\usepackage{enumerate}

\numberwithin{equation}{section}
\theoremstyle{plain}
\newtheorem{theorem}[equation]{Theorem}

\newtheorem{lemma}[equation]{Lemma}

\newtheorem{proposition}[equation]{Proposition}

\newtheorem*{namedtheorem}{\theoremname}
\newcommand{\theoremname}{testing}

\theoremstyle{definition}
\newtheorem{definition}[equation]{Definition}

\newcommand{\RR}{{\mathbb{R}}}
\newcommand{\ZZ}{{\mathbb{Z}}}

\newcommand{\CC}{{\mathbb{C}}}

\newcommand{\Hom}{\operatorname{Hom}}

\newcommand{\gLie}{\mathfrak{g}}
\newcommand{\Arith}{\operatorname{Ar}}

\newcommand{\Graph}{\mathcal{G}}

\newcommand{\Diag}{D}
\newcommand{\Tree}{\mathcal{T}}
\newcommand{\Path}{\mathcal{P}}

\renewcommand{\setminus}{{\smallsetminus}}

\newcommand{\tw}{\operatorname{tw}}

\newcommand{\cng}{\operatorname{cng}}
\newcommand{\cut}{\operatorname{w}}
\newcommand{\unit}{\mathbb{1}}

\newcommand{\Ring}{\mathcal{R}}
\newcommand{\id}{\operatorname{id}}

\newcommand{\Cat}{\mathcal{C}}
\newcommand{\cw}{\operatorname{cw}}
\newcommand{\poly}{\mathrm{poly}}

\newcommand{\parag}[1]{
	\medskip 
	{\bf #1}}

\usepackage{lineno}

\title{Parameterized Complexity of Quantum Invariants}
\author{Cl\'ement Maria}

\address{INRIA Sophia Antipolis-M\'editerran\'ee, FRANCE}
\email{clement.maria@inria.fr}

\begin{document}

\begin{abstract}
We give a general fixed parameter tractable algorithm to compute quantum invariants of links presented by diagrams, whose complexity is singly exponential in the carving-width (or the tree-width) of the diagram. 

In particular, we get a $O(N^{\frac{3}{2} \cw} \poly(n))$ time algorithm to compute any Reshetikhin-Turaev invariant---derived from a simple Lie algebra $\gLie$---of a link presented by a planar diagram with $n$ crossings and carving-width $\cw$, and whose components are coloured with $\gLie$-modules of dimension at most $N$. For example, this includes the $N^{\mathrm{th}}$-coloured Jones polynomial and the $N^{\mathrm{th}}$-coloured HOMFLYPT polynomial.
\end{abstract}

\maketitle

\section{Introduction}

In geometric topology, testing the topological equivalence of knots (up to isotopy) is a fundamental yet remarkably difficult algorithmic problem. 

A main approach is to compare knots by properties depending on their topological types only, called \emph{invariants}. Starting with the introduction by Jones~\cite{jones_poly85} in the 1985 of a new polynomial invariant of knots, we have witnessed the birth of a new domain of low dimensional topology called {\em quantum topology}. From the study of quantum groups~\cite{Drinfeld1986,Jimbo1985} in algebra, topologists have designed new families of topological invariants for knots, links, and $3$-manifolds, such as the Reshetikhin-Turaev invariants~\cite{Reshetikhin90}. 
In practice, these {\em quantum invariants} have shown outstanding discriminative properties for non-equivalent knots and links, e.g., in the composition of knot censuses, and are at the heart of deep mathematical conjectures in the field~\cite{Garoufalidis04,Garoufalidis11,Kashaev97,Murakami01}. 

Consequently, efficient algorithms to compute quantum invariants are of strong interest. However, even the simplest quantum invariants, such as the Jones polynomial~\cite{Jaeger90}, are \#P-hard to compute. 
A successful approach towards practical implementations has been the introduction of \emph{parameterized complexity} to low dimensional topology. Independently, computing the Jones polynomial~\cite{Makowsky03} and the HOMFLYPT polynomial~\cite{Burton18} have been shown to admit fixed parameter tractable algorithms in the tree-width of the input link diagrams. Note that similar techniques apply to 3-manifold quantum invariants, such as the Barrett-Westbury-Turaev-Viro invariants~\cite{BurtonMS18} of triangulated $3$-manifolds. These algorithms led to significant speed-ups in practice.

\parag{Contribution.} In this article, we give an algorithm to compute quantum invariants derived from ribbon categories~\cite{Reshetikhin90,turaev10-book}, taking into account the carving-width of the input link diagram.

\begin{theorem}\label{thm:maing}
Fix a strict ribbon category $\Cat$ of $\ZZ[q]$-modules, and free modules $V_1, \ldots, V_m \in \Cat$ of dimension bounded by $N$. The problem:

\bigskip

\begin{tabular}{|l}
{\sc Quantum invariant at $\Cat, V_1, \ldots, V_m$:}\\
{\bf Input}: $m$-components link $L$, presented by a diagram $\Diag(L)$,\\
{\bf Output}: quantum invariant $J_L^{\Cat}(V_1, \ldots, V_m)$ \\
\end{tabular}

\bigskip

\noindent
can be solved in $O(\poly(n) N^{\frac{3}{2} \cw}) = O(\poly(n) N^{\frac{3}{2} \sqrt{n}})$ machine operations, with $O(N^{\cw} + n)$ memory words, where $n$ and $\cw$ are respectively the number of crossings and the carving-width of the diagram $\Diag(L)$.
\end{theorem}

In particular, this implies that, up to some preprocessing normalisation, computing any Reshetikhin-Turaev invariant derived from a simple Lie algebra $\gLie$ is \emph{fixed parameter tractable} (complexity class {\tt FPT}) in the carving-width of the input link diagram. Cases of interests are, in particular, $\gLie=\mathfrak{sl}(2,\CC)$ giving the $N^{\text{th}}$-coloured Jones polynomials, and $\gLie=\mathfrak{sl}(n,\CC)$ giving the $N^{\text{th}}$-coloured HOMFLYPT polynomials. 
This algorithm is:
\begin{enumerate}[1]
\item the first fixed parameter tractable algorithm, and---considering $\cw = O(\sqrt{n})$---sub-exponen-tial time algorithm, for quantum invariants of knots stated in such generality (previously known cases were the (uncoloured) Jones polynomial~\cite{Makowsky03}, and the (uncoloured) HOMFLYPT polynomial~\cite{Burton18}), 
\item an exponential improvement over Burton's $2^{O(\cw \log \cw)} \poly(n)$ time algorithm for the uncoloured HOMFLYPT polynomial~\cite{Burton18}, and generally a low exponent ($\frac{3}{2}$) singly exponential algorithm for quantum invariants\footnote{Note that previous algorithms~\cite{Makowsky03} are expressed in terms of tree-width, which is proportional to the carving-width, in consequence exponents are not directly comparable.}.
\end{enumerate}

In Section~\ref{sec:background} we recall the definition of quantum invariants derived from ribbon categories, and notions of parameterized complexity. In Section~\ref{sec:graphicalalgo} we introduce a high-level parameterized algorithm based on graphical calculus and a tree embedding, then detail in Section~\ref{sec:mainop} the main operation of the algorithm. In Section~\ref{sec:algebraic} we develop the implementation of the algorithm in the case of a ribbon category of $\Ring$-modules, and analyse its arithmetic complexity in Section~\ref{sec:arithm}, in the case $\Ring = \ZZ[q]$. This last study implies that, when the type of invariant is part of the input, computing a quantum invariant is in the complexity class {\tt XP}.

\section{Background}
\label{sec:background}

We introduce the necessary notions from knot theory, quantum topology, and parameterized complexity.

\parag{Tangles and diagrams.} A \emph{tangle} is a piecewise linear embedding of a collection of arcs and circles into $\RR^2 \times [0,1]$, such that the arcs endpoints, called \emph{bases}, belong to the top or bottom boundaries $\RR^2 \times \{0\}$ and $\RR^2 \times \{1\}$. A tangle intersecting $i$ times $\RR^2 \times \{0\}$ and $j$ times $\RR^2 \times \{1\}$ is an $(i,j)$-tangle.

A \emph{link} is a tangle whose connected components are all closed curves (a $(0,0)$-tangle), and a \emph{knot} is a $1$-component link. We also consider link diagrams on the sphere $S^2$. An \emph{orientation} on a tangle is an orientation of each tangle component. 
Two tangles are equivalent iff they differ by an ambient isotopy of $\RR^2 \times [0,1]$ maintaining the boundary fixed.

A tangle \emph{diagram} is a projection of the tangle into the plane, induced by a projection of $\RR^2 \times [0,1]$ into $\RR \times [0,1]$, sending $\RR^2 \times \{0\}$ and $\RR^2 \times \{1\}$ to $\RR \times \{0\}$ and $\RR \times \{1\}$ respectively. In a tangle diagram, the only multiple points are \emph{crossings}, at which one section of the tangle crosses under or over another one transversally. 

Component orientations are pictured with arrow heads, and a $k \in \ZZ$ framing is pictured by $k$ \emph{positive twists} if $k > 0$, and $k$ \emph{negative twists} is $k < 0$. See Figure~\ref{fig:diagram}.

We refer to~\cite{lickorish1997-book} for more details on knot theory. 

\begin{figure}[t]
\centering
\includegraphics[width=9cm]{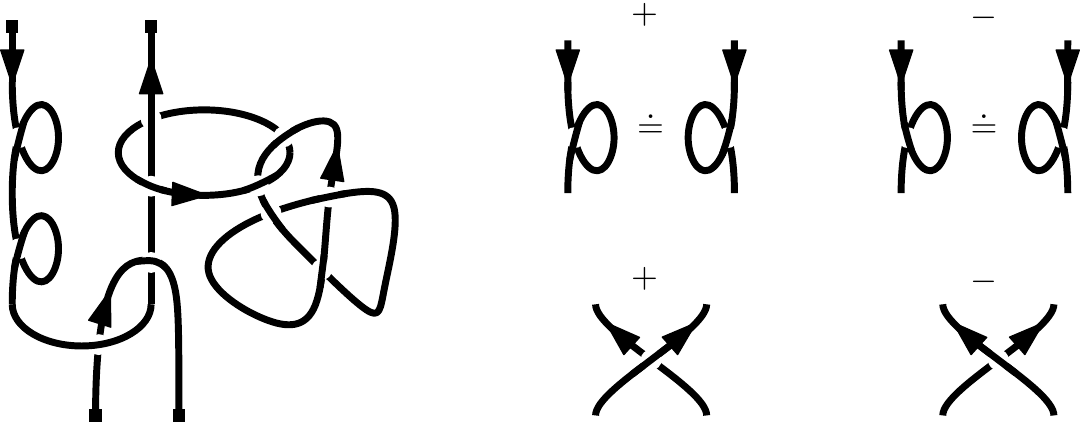}
\caption{Left: Diagram of a $4$-components oriented framed tangle, whose top left component has framing $+2$. Right: Positive/negative twists and crossings. The $\doteq$ symbol is an equivalence of diagrams.}
\label{fig:diagram}
\end{figure}

\parag{Ribbon categories and quantum invariants.} 
We refer to Turaev's monograph~\cite{turaev10-book} for the categorical formulation of quantum invariants. We only introduce the necessary notions.

Intuitively, a strict ribbon category is an abstraction of the category of modules over a commutative ring, with their usual tensor product. Some morphisms---called \emph{braidings}, \emph{twist}, \emph{evaluation and co-evaluation}---are distinguished in order to establish a connection between topology (tangles and knots) and algebra, via graphical calculus. 

More precisely, a \emph{strict ribbon category} $\Cat$ is a category with a unit object $\unit$ and which is equipped, for any objects $U,V,U',V'$ and morphisms $f \colon V \to V'$, $g\colon U \to U'$, with:

\begin{enumerate}[(a)]
\item an associative tensor product assigning to $U$ and $V$ an object $U \otimes V$, and to $f$ and $g$ a morphism $f \otimes g \colon U \otimes V \to U' \otimes V'$,
\item a natural braiding isomorphism $c_{U,V} \colon U \otimes V \to V \otimes U$,
\item a duality associating to any $V$ a dual object $V^*$, together with {\em co-evaluation} morphisms $b_V \colon \unit \to V \otimes V^*$ and {\em evaluation} morphisms $d_V \colon V^* \otimes V \to \unit$,
\item a natural twist isomorphism $\theta_V \colon V \to V$,
\item and where $\Hom_{\Cat}(\unit,\unit)$ has the structure of a commutative ring $\Ring$.
\end{enumerate}

By convention, the ``tensor product of zero objects'' is equal to $\unit$. 
In a strict ribbon category, these objects and morphisms satisfy additional compatibility constraints, that are necessary to state Theorem~\ref{thm:invsphere} below. 

For example, the category of modules over a commutative ring $\Ring$ with standard tensor product, and equipped with the trivial braiding $u \otimes v \mapsto v \otimes u$, forms a strict ribbon category. In this case, the ring $\Ring$, seen as a module over itself, is the unit object $\unit$, and any morphism $\Ring \to \Ring$ is a multiplication by a scalar $\tau \in \Ring$. Hence $\Hom_{\Cat}(\unit,\unit)$ is isomorphic to the commutative ring $\Ring$ itself. 
For invariants derived from quantum groups, we mainly focus on the category of $\Ring$-modules, generally free of finite dimension but with more complex braidings than the trivial ones. The ring $\Ring$ is $\ZZ[q]$ (up to normalisation), the ring of one-variable polynomials with integer coefficients. Morphisms between free modules are represented by matrices with $\Ring$-coefficients.

\parag{Graphical calculus and coloured tangles.} 
Fix a strict ribbon category $\Cat$. A \emph{colouring} of a link $L$, with $m$ ordered components $L_1, \ldots ,L_m$, is an assignment of an object $V_i \in \Cat$, $1 \leq i \leq m$, to every component $L_i$ of $L$.

A link diagram is considered in \emph{standard form} if it can be decomposed into the following pieces, described in Figure~\ref{fig:penrosefunctor}: $(i)$ vertical strands, $(vi)$ \& $(vii)$ positive and negative crossings, $(viii)$ \& $(ix)$ positive and negative right twists, and $(x)$ \& $(xi)$ caps and cups. See Figure~\ref{fig:ex_penrosefunctor} for a Hopf link in standard position. Any link (or tangle) diagram can be moved into standard form.

Rules $(i)$ to $(xi)$ of Figure~\ref{fig:penrosefunctor} gives the conversion from coloured tangle to $\Cat$-morphism, called {\em Penrose functor}. Specifically, given a coloured link diagram $\Diag(L)$, the Penrose functor turns the diagram into a morphism, following the rules:
\begin{description}
\itemsep0.5em 
\item[(o)] A morphism $f \colon U \to V$ in $\Cat$ is represented graphically by a box, aligned with $x$- and $y-$axis, called \emph{coupon}, with incoming vertical $V$-coloured strands (top) and outgoing vertical $U$-coloured strand (bottom),
\item[(i)] reversing a component orientation changes a colour $V$ to its dual $V^*$,
\item[(ii)] two parallel strands coloured $U$ and $V$ are equivalent to a single strand coloured $U \otimes V$,
\item[(iii)] a vertical strand coloured $V$ is equivalent to the identity morphism $\id_V$,
\item[(iv)] a morphism $g$ above another one $f$ is equivalent to there composition $g \circ f$,
\item[(v)] two morphisms $h_1$ and $h_2$ side by side are equivalent to their tensor product $h_1 \otimes h_2$,
\item[(vi) \& (vii)] a positive crossing is equivalent to a braiding morphism, a negative crossing is equivalent to the inverse of the braiding morphism,
\item[(viii) \& (ix)] positive and negative twists are equivalent to the twist morphism and its inverse respectively,
\item[(x) \& (xi)] caps and cups are equivalent to evaluation and co-evaluation respectively. 
\item[(xii)] the dual morphism $f^* \colon V^* \to U^*$ of a morphism $f \colon U \to V$ is given by the graphical equation (xii) or, equivalently, by:
\[
	f^* = (d_v \otimes \id_{U^*})\circ (\id_{V^*} \otimes f \otimes \id_{U^*})\circ (\id_{V^*} \otimes b_U).
\]
\end{description}

The morphisms are applied to the objects colouring the entering and leaving strands. Figure~\ref{fig:penrosefunctor} gives the morphism associated to the Hopf link coloured with objects $U$ and $V$.

Consequently, for a category $\Cat$, the Penrose functor associates to any coloured link a morphism $\unit \to \unit$. More generally, it associates to a coloured $(i,j)$-tangle a morphism $U_1 \otimes \ldots \otimes U_i \to V_1 \otimes \ldots \otimes V_j$, for the $U$s and $V$s colouring the bottom and top bases respectively. 

If the ordered components of a link $L$ are coloured $V_1, \ldots, V_m$, this morphism is written:
\[
	J^{\Cat}_L(V_1, \ldots, V_m) \in \Hom_{\Cat}(\unit,\unit).
\]

\begin{figure}[t]
\centering
\includegraphics[width=1.0\textwidth]{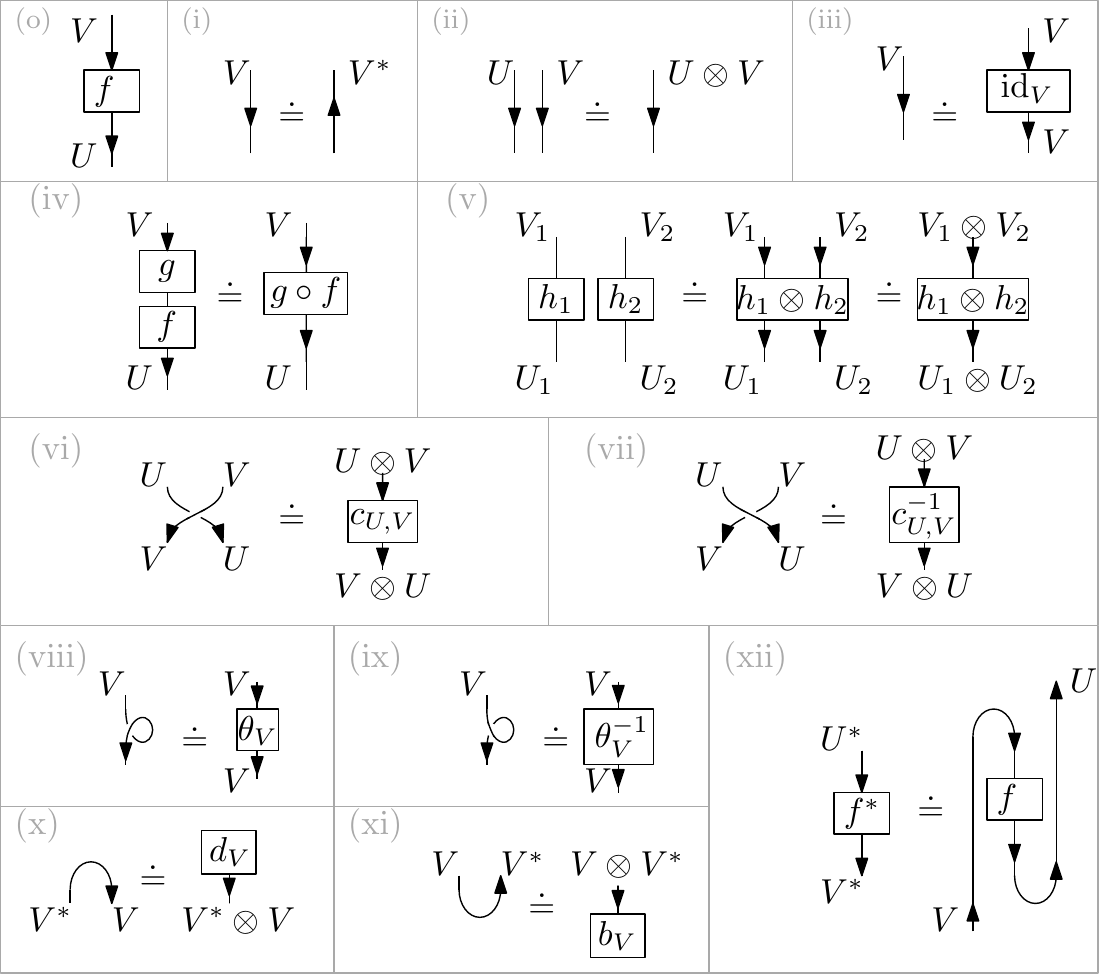}
\caption{Graphical calculus induced by Penrose functor.}
\label{fig:penrosefunctor}
\end{figure}

\begin{figure}[t]
\centering
\includegraphics[width=13.0cm]{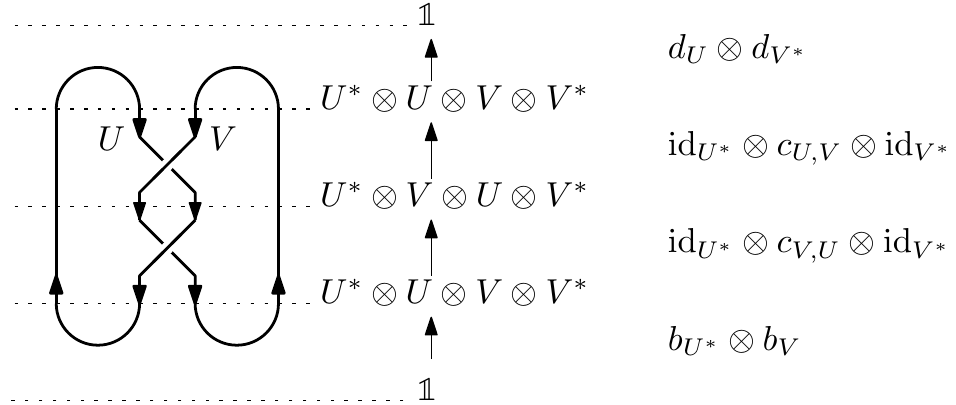}
\caption{Application of Penrose functor to the Hopf link coloured by objects $U$ and $V$ from a strict ribbon category, leading to a $\unit \to \unit$ morphism by composition.}
\label{fig:ex_penrosefunctor}
\end{figure}

Strict ribbon category produce topological invariants, called \emph{quantum invariants}:

\begin{theorem}[{\cite{Reshetikhin90,turaev10-book}}]\label{thm:invsphere}
Let $\Diag(L)$ be a diagram of an $m$-components link $L$ on $S^2$, and let $\Cat$ be a strict ribbon category. Let $V_1, \ldots, V_m$ be a colouring of the components of $L$. The quantity $J_{L}^{\Cat}(V_1, ..., V_m)$ produced by the Penrose functor is invariant by ambient isotopy of $S^2$ and Reidemeister moves on $\Diag(L)$. It is consequently a topological invariant of the coloured link $L$.
\end{theorem}

When $\Cat$ is the category of $\Ring$-modules, $J_{L}^{\Cat}(V_1, ..., V_k) \in \Hom(\unit,\unit) \cong \Ring$ is identified to a scalar in $\Ring$.

\parag{Graph parameters.} 
The \emph{carving-width}, also known as \emph{congestion}, is a graph parameter introduced by Seymour and Thomas~\cite{Seymour1994}.

\begin{definition}\label{Def:Carvingwidth}
Let $\Graph = (V,E)$ be a graph on $n$ vertices, with loops and multiple edges. Let $\Tree$ be an unrooted binary tree, with all internal nodes of degree $3$, and with $n$ leaves. 

An {\em embedding} $\phi$ of $\Graph$ into $\Tree$ is a bijective mapping between the nodes of $\Graph$ and the leaves of $\Tree$. Every edge $e$ of $\Tree$ induces a partition of the vertices of $\Graph$ into two sets, $V = U_e \sqcup V_e$, inherited from the partition of $\Tree \setminus e$ into two trees. Let $\cut(e)$ denote the number of edges in $\Graph$ between $U_e$ and $V_e$, called the weight of $e$.

The {\em congestion} of an embedding $(\Tree,\phi)$ is the maximal weight of a tree edge:
\[
	\cng(\Tree,\phi) = \max_{e \ \text{edge of} \ \Tree} \cut(e),
\]

The {\em carving-width} $\cng(\Graph)$ of a graph $\Graph$ is the minimal congestion over all its embeddings into binary trees. 
The \emph{carving-width $\cng(\Diag(L))$ of a link diagram $\Diag(L)$} is the carving-width of the $4$-valent planar graph it realises. The carving-width $\cng(L)$ of a link $L$ is the minimal carving-width of any of its diagrams.
\end{definition}

The carving-width of a graph is closely related to its \emph{tree-width}~\cite{robertson86-algorithmic}, which plays a major role in combinatorial algorithms. 

\begin{theorem}[Theorem~1 of \cite{DBLP:journals/jct/Bienstock90}]
\label{Thm:boundcngtw}
Let $\Graph$ be a graph of maximal degree $\delta$. Then,
\[
  \frac{2}{3} (\tw(\Graph) +1) \leq \cng(\Graph) \leq \delta (\tw(\Graph) +1).
\] 
\end{theorem}

Carving-width has however several advantages over tree-width. Notably, the former has been successfully used in low dimensional topology~\cite{HuszarS19,HuszarSW18,MariaP19,SchleimerdMPS19}.

First, for planar graphs---such as link diagrams---an optimal tree embedding realising the carving-width is polynomial time computable~\cite{Seymour1994}, when no efficient exact algorithm is known for computing an optimal tree decomposition. Second, optimal tree embeddings of planar graphs can be realised topologically, as stated below.

A \emph{bridge} in a connected graph $G$ is an edge of $G$ whose removal splits $G$ into more than one connected component. A tree embedding $(\Tree,\phi)$ of $G$ is \emph{bond} if the two vertex sets $U_e$ and $V_e$ from the cut associated to an edge $e$ of $\Tree$ induce connected sub-graphs in $G$.

\begin{theorem}[{\cite[Theorem 5.1]{Seymour1994}}]\label{thm:bridge}
Let $G$ be a simple connected bridgeless graph with more than two vertices. If $G$ has carving-width $\cw$ then there exists a bond tree embedding of $G$ of width $\cw$.
\end{theorem}

Up to a subdivision of multiple edges, which does not increase carving-width, a link diagram can be made simple, as a graph. Being $4$-valent, it is bridgeless, and, if connected, it consequently admits a bond tree embedding of minimal congestion. We interpret a bond tree embedding of a planar graph (on the sphere $S^2$) as a collection of disjoint Jordan curves $\lambda_e \subset S^2$, one for each edge $e$ of $\Tree$, realising the cut $U_e \sqcup V_e$~\cite{SchleimerdMPS19}.

For planar graphs, a bond tree embedding of minimal congestion can be computed in polynomial time~\cite{Gu:2008:OBP:1367064.1367070,Seymour1994}.


\section{Fixed parameter tractable algorithm via graphical calculus}
\label{sec:graphicalalgo}

Let $\Cat$ be a strict ribbon category, and let $L$ be an oriented link with $m$ components $L_1, \ldots, L_m$. Let $\Diag(L)$ be an oriented link diagram of $L$, where each link component $L_i$ is coloured by an object $V_i$ from the category $\Cat$, such that the Penrose functor gives an isotopy invariant of $L$ associated to its colouring, as described in Theorem~\ref{thm:invsphere}.

Without loss of generality, we assume that the diagram $\Diag(L)$ is connected as a graph, and has at least $2$ crossings. It follows from the definition of Penrose functor that the quantum invariant of a separable link $L \cup L'$ is the product of the invariants of $L$ and $L'$, and they can be computed separately.  


\subsection{Tree embedding of link diagrams.} 

Let $(\Tree, \phi)$ be a bond tree embedding of the planar graph of $\Diag(L)$, and root it by subdividing an arbitrary tree edge, picking the centre as the root. All edges of $\Tree$ have now a parent and child endpoint. By convention, we add a ``half-edge'' on top of the tree, having the root as child. Every inner node in $\Tree$ has consequently degree $3$, with two edges ``going down'', and one edge ``going up''.

Let $e$ be an edge of $\Tree$ with child node $x$, and $X$ the set of crossings mapped to the leaves of the subtree $\Tree_x$ rooted at $x$. 
According to Theorem~\ref{thm:bridge}, there exists a Jordan curve $\lambda_e$ separating $X$ from the rest of the diagram. The diagram being on the sphere, we draw the tangle ``inside'' the Jordan curve when we represent it on the plane. 

To edge $e$ corresponds a $(0,\cut(e))$-tangle $T$, spanned by the crossings $X$ and contained ``inside'' $\lambda_e$. We mark an arbitrary but fixed ``bullet'' point on $\lambda_e$ and order the bases of $T$ counter-clockwise. We get a $(0,\cut(e))$-tangle by isotopically sliding all bases to the top boundary, such that the first base in the bullet ordering is rightmost on the top boundary. See Figure~\ref{fig:atom} for examples of $(0,\cut(e))$-tangles at the tree leaves, and Figure~\ref{fig:spin} (Left) for bases ordered by bullet ordering.

In the process of the algorithm below, bullet orderings are assigned on the fly.


\subsection{Tree traversal algorithm.} 
Let $\Diag(L)$ be coloured by objects of the category $\Cat$. To every edge $e$ of weight $\cut(e)$ in $\Tree$, the Penrose functor assigns a $\Cat$-morphism $f_e \colon \unit \to V_1 \otimes \ldots \otimes V_{\cut(e)}$ to the associated tangle, where $V_1, \ldots , V_{\cut(e)}$ are the colours of the strands intersecting the Jordan curve $\lambda_e$.

The morphism associated to the half-edge at the root is a $\unit \to \unit$ morphism, because the corresponding Jordan curve does not intersect the link diagram. This morphisms gives the invariant $J_L^\Cat \in \Ring$ of Theorem~\ref{thm:invsphere}. 
All edge morphisms are computed recursively following a depth first traversal of $\Tree$. We describe the base morphisms assigned to the edges whose child node is a leaf, and we describe an algorithm for inner edges in the next section.

\begin{figure}
\centering
\includegraphics[width=12cm]{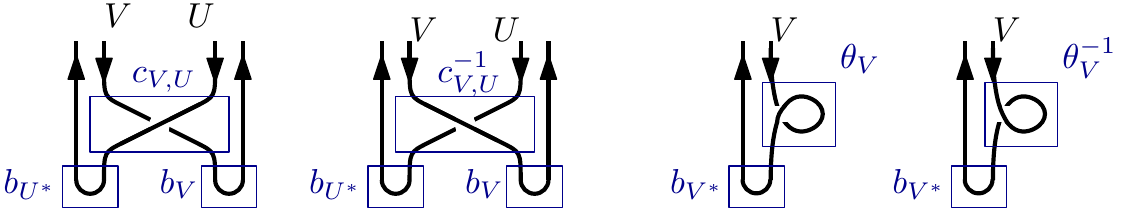}
\caption{The four tangles and associated morphisms at the tree leaves. From left to right: Equations~(\ref{eq:atom1}),~(\ref{eq:atom2}),~(\ref{eq:atom3}), and~(\ref{eq:atom4}). The marked bullet point is on the left of each diagram, and is selected such that only these four morphisms are encountered.}
\label{fig:atom}
\end{figure}

\subsection{Morphisms at the leaves.}
\label{subsec:leaves}

Up to reorientation of the strands, which algebraically consists of dualising colours, we can restrict to four base morphisms:

\begin{minipage}{0.45\textwidth}
  \begin{equation}\label{eq:atom1}
    (\id_{U^*} \otimes c_{V,U} \otimes \id_{V^*}) \circ (b_{U^*} \otimes b_V)
  \end{equation}
\end{minipage}
\hfill\begin{minipage}{0.45\textwidth}
  \begin{equation}\label{eq:atom2}
    (\id_{U^*} \otimes c_{V,U}^{-1} \otimes \id_{V^*}) \circ (b_{U^*} \otimes b_V)
  \end{equation}
\end{minipage}

\begin{minipage}{0.45\textwidth}
  \begin{equation}\label{eq:atom3}
    (\id_{V^*} \otimes \theta_V) \circ b_{V^*}
  \end{equation}
\end{minipage}
\hfill\begin{minipage}{0.45\textwidth}
  \begin{equation}\label{eq:atom4}
    (\id_{V^*} \otimes \theta_V^{-1}) \circ b_{V^*}
  \end{equation}
\end{minipage}

\medskip

They correspond graphically to the diagrams in Figure~\ref{fig:atom}, where the bullet ordering is chosen to restrict to these four cases.


\subsection{Merging morphisms at tree nodes.} 
Every inner node $x$ of $\Tree$ is the parent node of two edges $e_1$ and $e_2$, and the child of an edge $e$. Given the morphisms $f_{e_1}$ and $f_{e_2}$ for edges $e_1$ and $e_2$ respectively, we construct the morphism $f_e$ for edge $e$.

First, note that the bullet ordering of the strands intersecting $\lambda_{e_1}$ and $\lambda_{e_2}$ leads to three configurations when representing morphisms $f_{e_1}$ and $f_{e_2}$ with coupons ; see Figure~\ref{fig:spin} where thick lines represent sets of parallel tangle strands. By hypothesis, morphisms on tree edges have domain $\unit$. The coupons for $f_{e_1}$, $f_{e_2}$, and $f_e$ (the outer coupon) are obtained by a plane isotopy forcing the strands to intersect coupons on their top side, and putting bullets on the coupons' left sides. The bullet of the outer coupon $f_e$ is selected so as to restrict to the three configurations of Figure~\ref{fig:spin}.

\begin{figure}[t]
\centering
\includegraphics[width=14cm]{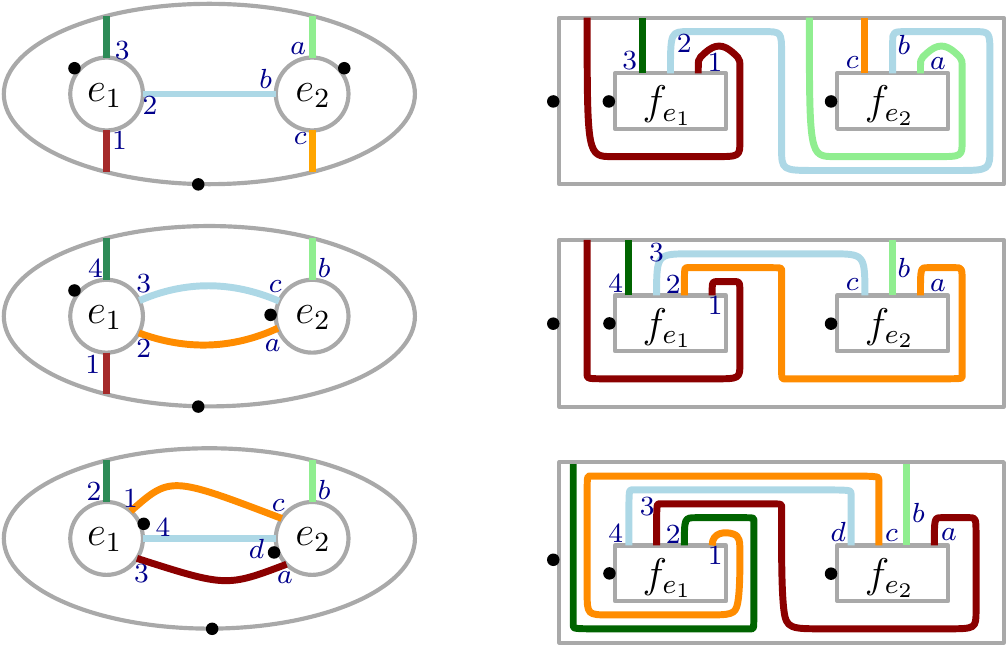}
\caption{Merging two sub-trees. Left: Planar embeddings of the diagram with Jordan curves $\lambda_{e_1}$, $\lambda_{e_2}$ (inner circles) and $\lambda_e$ (outer circle), depending on the position of the bullets for $\lambda_{e_1}$ and $\lambda_{e_2}$. The bold lines connecting the Jordan curves represent multiple parallel strands connecting the corresponding tangles. Right: Coupons for $f_{e_1}$, $f_{e_2}$ and $f_e$ (outer coupon) obtained after plane isotopy. The bullet for $\lambda_e$ is selected so as to restrict to these three cases.}
\label{fig:spin}
\end{figure}


\section{Factorisation of morphisms at tree nodes} 
\label{sec:mainop}
Given the morphisms $f_{e_1}$ and $f_{e_2}$ in Figure~\ref{fig:spin}, we describe graphically a factorisation scheme to obtain the morphism $f_e$. 

\subsection{Sliding and canonical form.}
The canonical form for morphisms to be merged is depicted in the top left corner of Figure~\ref{fig:proof2}. It consists of two side-by-side morphisms $g_1$ and $g_2$, bridged by parallel strands coloured $U_1, \ldots U_k$. All other strands go vertically.

Given morphisms $f_{e_1}$ and $f_{e_2}$ in Figure~\ref{fig:spin}, we obtain a canonical form by sliding strands, wrapping clockwise around the coupons, under the coupons. For example, in the top right case of Figure~\ref{fig:spin}, we slide strand $1$ under the $f_{e_1}$-coupon, and strands $a$ and $b$ under the $f_{e_2}$-coupons.

The details of the operation are depicted in Figure~\ref{fig:twist}, where the $V$-strand wraps clockwise around the $f$-coupon, and $f$ is a $\unit \to U \otimes V$ morphism. Sliding the $V$-strand under the coupon by tangle isotopy produces a positive twist $\theta_V$ and a positive crossing $c_{V,U}$.

Decomposing further in Figure~\ref{fig:twist}, let $U = U_i \otimes \ldots \otimes U_1$ be the tensor product of the colours of $i$ parallel strands, and $V = V_j \otimes \ldots \otimes V_1$ the tensor product of $j$ parallel strands wrapping clockwise around the $f$ coupon. As depicted in the figure, sliding the $j$ strands under $f$ induces 
\begin{enumerate}[-]
\item a twist $\theta_{V_\ell}$ on each of the $V_\ell$-coloured strands, $1 \leq \ell \leq j$,
\item a sequence of $j(j-1)$ positive and negative crossings of type $c^\pm_{V_\ell,V_k}$, followed by
\item a sequence of $ij$ positive crossings of type $c_{V_\ell,U_k}$.
\end{enumerate}

We obtain the morphisms $g_1, g_2$ of the canonical form (Figure~\ref{fig:proof2}) by factorising the morphisms $f_{e_1}$ and $f_{e_2}$ with these sequences of twists and crossings, after the sliding operation.

\begin{figure}[t]
\centering
\includegraphics[width=14cm]{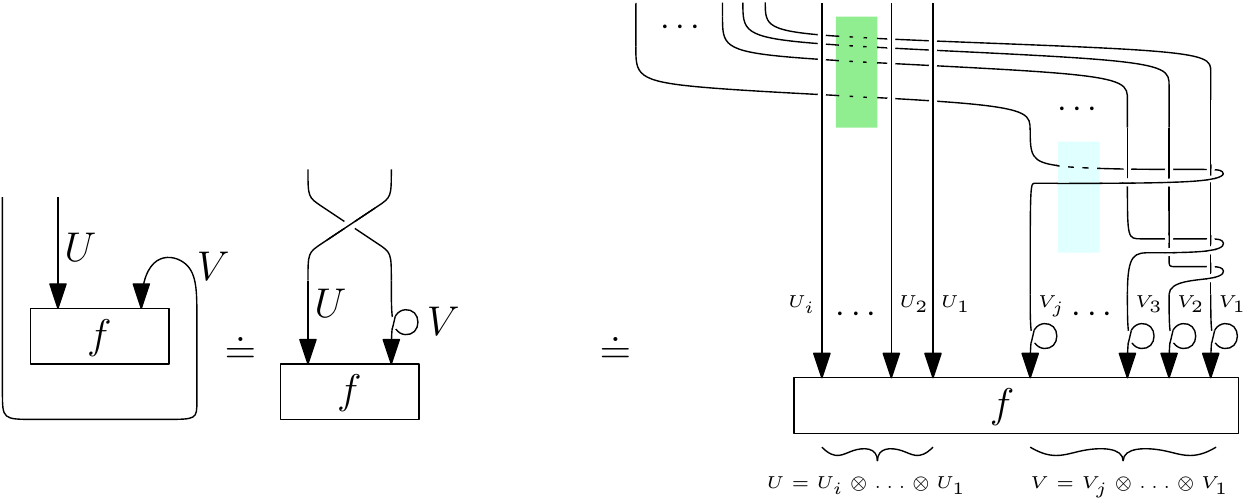}
\caption{Sliding of a $V = V_j \otimes \ldots \otimes V_1$-coloured strand under an $f$-coupon by underlying knot isotopy. The operation composes $f$ with a consecutive sequence of $j$ twists $\theta_{V_\ell}$, of $j(j-1)$ crossings $c^{\pm}_{V_i,V_j}$, and $ij$ crossings $c^{\pm}_{U_k,V_\ell}$.}
\label{fig:twist}
\end{figure}


\begin{figure}[tp]
\centering
\includegraphics[width=14cm]{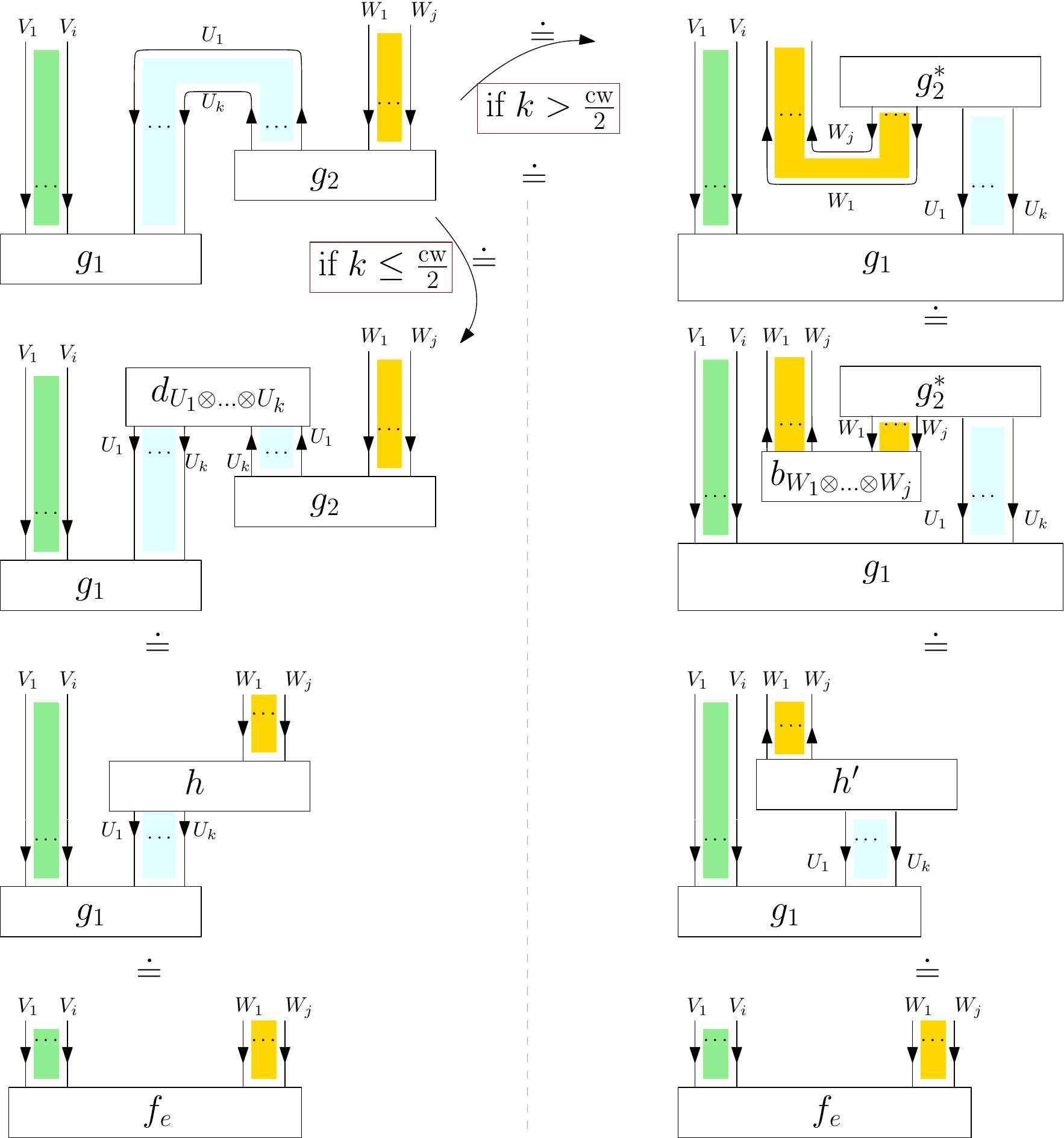}
\caption{Merging of two coupons in a canonical form (top left) along $k$ strands coloured $U_1, \ldots, U_k$. The factorisation scheme differs whether $k \leq \cw/2$ (left column) or $k > \cw / 2$ (right column). The top right equivalence comes from the equality in Figure~\ref{fig:equality2}.}
\label{fig:proof2}
\end{figure}

\subsection{Factorisation of the canonical form.}
Figure~\ref{fig:proof2} pictures two factorisation schemes for side-by-side morphisms $g_1$ and $g_2$ in canonical form, bridged by $k$ parallel strands coloured $U_1, \ldots, U_k$. Denote by $\cw$ the carving-width of the link diagram, and assume the tree embedding $(\Tree,\phi)$ has width $\cw$. We distinguish two cases:


\parag{Small bridge.} For $k$ smaller than half the carving-width (Figure~\ref{fig:proof2}, Left), we consider first the morphism $d_{U_1 \otimes \ldots \otimes U_k}$ induced by the composition of the evaluation morphisms $d_{U_\ell}$, $\ell = k \ldots 1$. More precisely, the morphism $d_{U_{1} \otimes \ldots \otimes U_k} \colon U_1 \otimes \ldots \otimes U_k \otimes U_k^* \otimes \ldots \otimes U_1^* \to \unit$, is obtained by composing the evaluation morphisms from bottom up:
\begin{equation}\label{eq:evaluations}
\begin{array}{ccl}
d_{U_{1} \otimes \ldots \otimes U_k} &\colon& U_1 \otimes \ldots \otimes U_k \otimes U^*_k \otimes \ldots \otimes U^*_1 \to \unit, \\
                                       &=& 
  \prod_{\ell=k}^1 \left( \id_{U_{1} \otimes \ldots \otimes U_{\ell-1}} \otimes d_{U^*_\ell} \otimes \id_{U^*_{\ell-1} \otimes \ldots \otimes U^*_1} \right) \\
  \end{array}
\end{equation}
where $\ell=k$ is the rightmost term of the composition.

The (partial) composition of $d_{U_{1} \otimes \ldots \otimes U_k}$ with $g_2$ through $U^*_k \otimes \ldots \otimes U^*_1$ gives the morphism $h$:
\begin{equation}\label{eq:h}
\begin{array}{ccl}
h &\colon& U_1 \otimes \ldots \otimes U_k \to W_1 \otimes \ldots \otimes W_j,\\
%
&=& \left(d_{U_1 \otimes \ldots \otimes U_k} \otimes \id_{W_1 \otimes \ldots \otimes W_j} \right) \circ \left( \id_{U_1 \otimes \ldots \otimes U_k} \otimes g_2 \right).\\
\end{array}
\end{equation}
Finally, the morphism $f_e$ obtained from the merging of $f_{e_1}$ and $f_{e_2}$ is given by the (partial) composition of $g_1$ and $h$, through $U_1 \otimes \ldots \otimes U_k$. Precisely,
\begin{equation}\label{fig:fe}
\begin{array}{ccl}
f_e &\colon& \unit \to V_1 \otimes \ldots \otimes V_i \otimes W_1 \otimes \ldots \otimes W_j,\\
&=& \left( \id_{V_1 \otimes \ldots \otimes V_i} \otimes h \right) \circ g_1.\\
\end{array}
\end{equation}

By construction, these operations give the morphism $f_e$ induced by the Penrose functor on the coloured tangle associated to the subtree of $\Tree$ rooted at the child node of edge $e$.


\begin{figure}
\centering
  \centering
  \includegraphics[width=12cm]{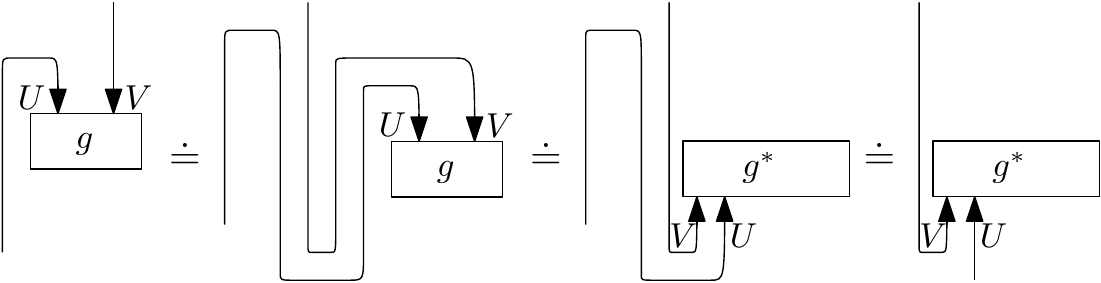}
  \captionof{figure}{Planar isotopy, then factorisation with $g^*$, the dual morphism to $g$.}
  \label{fig:equality2}
\end{figure}

\parag{Large bridge.} The case $k$ strictly larger than half the carving-width starts by flipping upside-down coupon $g_2$. Precisely, this operation is depicted in Figure~\ref{fig:equality2}. Starting with a morphism $g$, it consists of a planar isotopy to produce $g^*$, the dual morphism to $g$, then another planar isotopy. In the case where the category $\Cat$ satisfies the hypothesis of Theorem~\ref{thm:invsphere}, Figure~\ref{fig:equality2}, depicting an isotopy, proves the equality:
\[
	\left( d_U \otimes \id_V \right) \circ \left( \id_{U^*} \otimes g \right) = 
	\left( \id_V \otimes g^* \right)\circ \left( b_V \otimes \id_{U^*} \right)
\]

Applied to the canonical form on $g_1$ and $g_2$ (Figure~\ref{fig:proof2}, Top) the operation gives the composition of morphisms, involving $g_1$ and $g^*_2$, in the top right corner of Figure~\ref{fig:proof2}.

The following compositions are similar to the case of a small bridge. Morphism $b_{W_1 \otimes \ldots \otimes W_j}$ describes the composition of the co-evaluation morphisms for $W_1, \ldots, W_j$, i.e.,
\begin{equation}\label{eq:coevaluations}
\begin{array}{ccl}
b_{W_1 \otimes \ldots \otimes W_j} &\colon& \unit \to W_1^* \otimes \ldots \otimes W_j^* \otimes W_j \otimes \ldots \otimes W_1, \\
  &=& 
  \prod_{\ell=1}^j \left( \id_{W_{1}^* \otimes \ldots \otimes W_{\ell-1}^*} \otimes b_{V_\ell} \otimes \id_{W_{\ell-1} \otimes \ldots \otimes W_1} \right).\\
\end{array}
\end{equation}
where $\ell = 1$ is the rightmost term of the composition.

The morphism $h'$ is obtained by (partial) composition of $b_{W_1 \otimes \ldots \otimes W_j}$ and $g^*_2$:
\begin{equation}\label{eq:hprime}
h' =  \left( \id_{W_1^* \otimes \ldots \otimes W_j^*} \otimes g_2^* \right) \circ \left( b_{W_1 \otimes \ldots \otimes W_j} \otimes \id_{U_1 \otimes \ldots \otimes U_k} \right),
\end{equation}
and $f_e$ is obtained by (partial) composition of $g_1$ and $h'$:
\begin{equation}\label{eq:febis}
f_e = \left( \id_{V_1 \otimes \ldots \otimes V_i} \otimes h' \right) \circ g_1.
\end{equation}


\parag{Correctness.} The correctness of the algorithm follows directly from Theorem~\ref{thm:invsphere}, noting that the algorithm consists of an isotopy of the link, realisable by isotopies of the sphere on which the diagram is drawn, and Reidemeister moves.


\section{Algebraic implementation and complexity} 
\label{sec:algebraic}

For the implementation of the algorithm, we assume that the objects in the category $\Cat$ are finite dimensional free $\Ring$-modules, for a commutative ring with unity $\Ring$. Denote the dimension of every link component colour $V_i$ by $N_i := \dim V_i$, and let $N := \max_i \{ \dim V_i \}$. Fixing a basis for every $V_i$, all morphisms in $\Cat$---in particular the distinguished braiding, evaluation and co-evaluation, and twist morphisms---are represented by matrices with $\Ring$ coefficients.

This model is general, and contains in particular all quantum invariants derived from quantum groups.

\subsection{Elementary compositions.}

We consider the seven elementary compositions of morphisms depicted in Figure~\ref{fig:elementary}. They respectively represent the composition with $(1)$ a single braiding, $(2)$ a single twist, $(3)$ a single co-evaluation, $(4)$ a single evaluation. Cases $(5)$, $(6)$, and $(7)$ represent three types of partial compositions of the morphisms $f$ and $g$. We describe algorithms to perform these compositions on matrices.

\begin{lemma}\label{lem:elementary1}
Consider the elementary morphism compositions in Figure~\ref{fig:elementary}~(1),~(2),~(3), and~(4). Let $U,V,V',W$ be finite dimensional free $\Ring$-modules, with $\dim U = a$, $\dim V = b$, $\dim V' = b'$, and $\dim W = c$. Then, given the matrices for morphisms $f$, $\theta^{\pm}_V$, $c^{\pm}_{V,V'}$, $b_V$, and $d_U$, we can compute the matrix for morphism $h$ in: 
\begin{itemize}
\item $O(a(bb')^2c)$ arithmetic operations in $\Ring$ for (1), 
\item $O(ab^2c)$ arithmetic operations for (2) and (3), and 
\item $O(a^2b)$ arithmetic operations for (4).
\end{itemize}
The memory complexity of the operation does not exceed the size of the output, which is a row or column vector $h$ containing scalars from $\Ring$.
\end{lemma}

\begin{proof}
\noindent
{\bf Figure~\ref{fig:elementary}~(1),~(2), and~(3).} All three cases consist of the matrix-vector product $h = (\id_U \otimes M \otimes \id_W) \cdot f$, where $M$ is respectively the $(bb'\times bb')$-matrix $c^{\pm}_{V,V'}$, the $(b\times b)$-matrix $\theta^{\pm}_V$, and the $(b^2 \times 1)$-matrix $b_V$. 

Consider $M$ to be an $(m \times m')$-matrix, with coefficients $(M_{i,j})_{1 \leq i \leq m, 1 \leq j \leq m'}$. Matrix $(\id_U \otimes M \otimes \id_W)$ has at most $m'$ non-zero coefficients per row. We get the formula for the $i^{\text{th}}$ entry of $h$:
\[
	h_{i,1} = \sum_{k=1 \ldots m'} M_{\beta+1,k} \cdot f_{\alpha cm' + \gamma + (k-1)c,1}
\]
where $i$ is uniquely written as $i=\alpha\cdot cm + \beta \cdot c + \gamma$, with $0 \leq \alpha \leq a-1$, $0 \leq \beta \leq m-1$, and $1 \leq \gamma \leq c$. 
Computing $h$ requires $O(m' |f|)$ arithmetic operation in $\Ring$, where $|f|$ is the length of vector $f$, storing $O(|f|)$ scalars from $\Ring$.

\medskip

\noindent
{\bf Figure~\ref{fig:elementary}~(4).} With a similar approach, we get for any $j$, $1 \leq j \leq a^2b$:
\[
	h_{1,j} = (d_U)_{1,\alpha a + \gamma} \cdot f_{\beta+1,1}
\]
where $j$ is uniquely written as $j=\alpha \cdot ab + \beta \cdot a + \gamma$, with $0 \leq \alpha \leq a-1$, $0 \leq \beta \leq b-1$, and $1 \leq \gamma \leq a$. 
The algorithm has complexity $O(a^2b)$ and memory usage $O(a^2b)$.

\end{proof}

\begin{figure}[t]
\centering
\includegraphics[width=15cm]{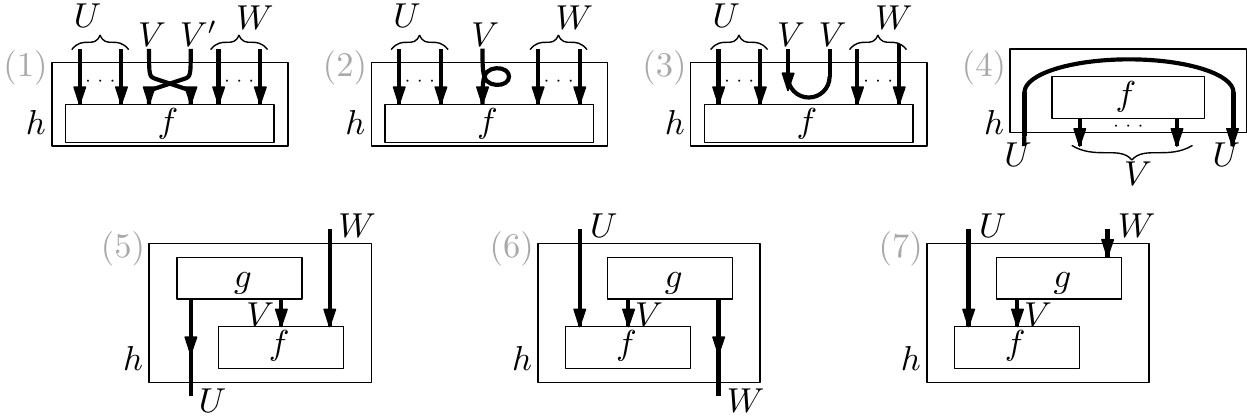}
\caption{Graphical representation of the seven elementary compositions of morphisms.}
\label{fig:elementary}
\end{figure}

\begin{lemma}\label{lem:elementary2}
Consider the elementary morphism compositions in Figure~\ref{fig:elementary}~(5),~(6), and~(7). Let $U,V,W$ be finite dimensional free $\Ring$-modules, with $\dim U = a$, $\dim V = b$, and $\dim W = c$. Then, given the matrices for morphisms $f$ and $g$, we can compute the matrix for morphism $h$ in $O(abc)$ arithmetic operations in $\Ring$, and memory complexity $O(ab+bc+ac)$ times the size of a scalar in $\Ring$.
\end{lemma}
 
\begin{proof}
\noindent
{\bf Figure~\ref{fig:elementary}~(5).} Morphism $f$ is a $bc \times 1$ matrix, and morphism $g$ is a $1 \times ab$ matrix. Define $h = (g \otimes \id_W)(\id_U \otimes f)$. Morphism $h$ is a $c \times a$ matrix.

Studying the shape of matrices $(\id_U \otimes f)$ and $(g \otimes \id_W)$, it appears that every one of the $c \times a$ coefficients of the product $h = (g \otimes \id_W)(\id_U \otimes f)$ is a sum of $O(b)$ terms. Precisely, an explicit computation gives us, for any $i,j$, $1 \leq i \leq c$, $1 \leq j \leq a$:
\[
	h_{i,j} = \sum_{k=0 \ldots b-1} g_{1,(j-1)b+k+1} \cdot f_{kc+i,1}
\]

Morphism $h$ is a $c \times a$ matrix, and each of its coefficients can be computed in $O(b)$ arithmetic operations in $\Ring$, leading to the $O(abc)$ time complexity. The memory consumption is the sum of the sizes of the input matrices $f$ and $g$, and the output matrix $h$.

\medskip

{\bf Figure~\ref{fig:elementary}~(6).} With a similar approach, for any $i,j$, $1 \leq i \leq a$, $1 \leq j \leq c$:
\[
	h_{i,j} = \sum_{k=0 \ldots b-1} g_{1,kc+j} \cdot f_{(i-1)b+k+1,1}
\]

\medskip

\noindent
{\bf Figure~\ref{fig:elementary}~(7).} With a similar approach, for any $i$, $1 \leq i \leq ac$, write $i = \alpha c + \beta$, for $0 \leq \alpha \leq a-1$ and $1 \leq \beta \leq c$:
\[
	h_{i,1} = \displaystyle \sum_{k=1 \ldots b} g_{\beta,k} f_{\alpha b + k,1}
\]
\end{proof}


\subsection{Implementation of the algorithm.} 

We implement the algorithm described in Sections~\ref{sec:graphicalalgo} and~\ref{sec:mainop} using the elementary composition of Figure~\ref{fig:elementary}. 
Define $N$ a bound on the dimension of the different modules $U_i$, $V_j$, $W_k$ colouring the components of the link.

\parag{Leaf morphisms.} The leaf morphisms described in Equations~(\ref{eq:atom1}-\ref{eq:atom1}) and Figure~\ref{fig:atom} are implemented using elementary compositions~(1) and~(2). By Lemma~\ref{lem:elementary1}, the complexity is at most $O(N^6)$ arithmetic operations in $\Ring$.

\parag{Sliding under a coupon.} The sliding operation as presented in Figure~\ref{fig:twist} composes a morphism $f$ with a sequence of twist and braiding morphisms.  
Precisely, let $h$ denote the entire morphism in Figure~\ref{fig:twist}. Starting from the $(O(N^{i+j}) \times 1)$ matrix $f$, it is computed iteratively applying $j$ times elementary composition~(2) for the twists, then $j(j-1)$ times elementary composition~(1) for the braidings between $V_i$- and $V_j$-strands, and finally $ij$ times elementary composition~(1) for the braidings between $V_i$- and $U_j$-strands.

During the computation, we maintain a vector of size $(1 \times O(N^{i+j}))$. Applying Lemma~\ref{lem:elementary1}, the sliding operation runs in $O(j(i+j) N^{i+j+2})$ arithmetic operations in $\Ring$, storing $O(N^{i+j})$ scalar from $\Ring$. 
In the algorithm, $i+j \leq \cw$, the carving-width of the link diagram. Consequently, we get $O(\cw^2 N^{\cw +2})$ operations, with memory $O(N^{\cw})$.


\parag{Construction of evaluations and co-evaluations.} The morphism $d_{U_1 \otimes \ldots \otimes U_k}$ appearing in Figure~\ref{fig:proof2} is the result of $k$  elementary compositions of type~(4). The morphisms maintained during the computation are of size $(1 \times O(N^{2 k}))$. Applying Lemma~\ref{lem:elementary1}, the computation takes a total of $O(k N^{2k})$ arithmetic operations in $\Ring$, storing $O(N^{2k})$ scalars from $\Ring$. 
The case $b_{W_1 \otimes \ldots \otimes W_j}$ is similar.

In the algorithm, $k$ (or $j$) is smaller than $\cw/2$. Consequently, the complexity is $O(\cw N^{\cw})$ arithmetic operations, storing $O(N^{\cw})$ scalars.

\parag{Composition of morphisms.} Finally, the compositions of morphisms described in Figure~\ref{fig:proof2} are implemented with a constant number of elementary compositions~(5),~(6), and~(7). Considering Lemma~\ref{lem:elementary2}, the product $abc$ of dimensions never exceed $N^{\frac{3}{2} \cw}$. Consequently, the compositions of Figure~\ref{fig:proof2} are implemented using $O(N^{\frac{3}{2} \cw})$ arithmetic operations in $\Ring$, storing $O(N^{\cw})$ scalars from $\Ring$.

\parag{Overall complexity.} In conclusion, we sum up the different steps of the algorithm and its implementation. Let $\Diag$ be a coloured link diagram with $n$ crossings and carving width $\cw$, where the dimension of each colouring module is at most $N$. The algorithm first computes an optimal tree embedding in $O(\poly(n))$ operations. The tree has size $n$ and width $\cw$. W.l.o.g., we assume the diagram has at least one crossing that is not a twist, and consequently $\cw \geq 4$. The quantum invariant associated to the colouring is computed in:
\[
	O(n N^{\frac{3}{2}\cw}) \ \text{arithmetic operations in} \ \Ring, 
\]
storing:
\[
   O(n) \ \text{words for the diagram, plus} \ O(N^{\cw}) \ \text{scalars from} \ \Ring. 
\]


\section{Arithmetic complexity and quantum invariants of links}
\label{sec:arithm}

Working with matrices with $\Ring$-coefficients, for a ring $\Ring$, allows the algorithm to be applied in great generality. For example, any complex simple Lie algebra $\gLie$ produces quantum invariants of links, that can be expressed as a composition of morphisms between free $\Ring$-modules, and to which our algorithm can be applied. See~\cite[Chapter 6]{turaev10-book} for an explicit construction. 

In this case, $\Ring$ is a polynomial ring, and both degrees of polynomials as well as values of coefficients may blow-up during intermediate computation. Specifically, implemented naively, both arithmetic operations within $\Ring$ and bit size of $\Ring$-elements may become {\em exponential in $n$}. 

In this section we describe a solution to control the arithmetic complexity in the case $\Ring = \ZZ[q]$, which is sufficient for all $J_{L}^{\gLie}$ invariants, up to normalisation. We also provide detailed complexity bounds for completeness.


\subsection{Arithmetic complexity of polynomial invariants.}

We give coarse, but general, bounds on the degrees and coefficients of a polynomial invariant produced by the algorithm introduced above, that are sufficient for the complexity analysis. 

\begin{proposition}
Let $\Cat$ be a strict ribbon category of $\ZZ[q]$-modules, and let $\Diag(L)$ be an $n$-crossings diagram of a link $L$ whose components are coloured with free modules $V_1, \ldots, V_m \in \Cat$, of dimension at most $N$.

Let $d_0$ and $C_0$ be respectively a bound on the degree and a bound on the absolute value of coefficients of all polynomials in the matrices $c^{\pm}_{V_i,V_j}$, $\theta^{\pm}_{V_i}$, $d_{U_i}$, $b_{U_i}$, for $1 \leq i,j \leq m$.

Then the polynomial invariant $J_L^{\Cat}(V_1, \ldots, V_m) \in \ZZ[q]$ has degree and absolute value of coefficients bounded by $d_n$ and $C_n$ respectively, with:
\[
	d_n = O(nd_0) \ \ \ \text{and} \ \ \ C_n = 2^{O(n \sqrt{n} \log N + n \log C_0)}.
\]
\end{proposition}

\begin{proof}
Consider a tree embedding of graph $\Diag(L)$ where the tree is a path, with leaves attached to it. The minimal congestion over all such embeddings is called the {\em cut-width} of the graph, and is $O(\sqrt{n})$ due to the planar separator theorem.

Let $k$ be the cut-width of $\Diag(L)$, and $(\Path,\phi)$ a minimal embedding of $\Diag(L)$ into a path-tree. 
Running the algorithm of Sections~\ref{sec:graphicalalgo}-\ref{sec:algebraic} on this path decomposition boils down to computing the product of $O(n)$ matrices:
\[
	M_{\alpha \cdot n} \cdot \ldots \cdot M_1
\]
where all matrices are tensor products of a $c^{\pm}_{V_i,V_j}$, $\theta^{\pm}_{V_i}$, $d_{U_i}$, $b_{U_i}$, for some $1 \leq i,j \leq m$, with identities, and all matrices have size at most $N^{O(k)} \times N^{O(k)}$. Additionally, $M_1$ has $1$ column, and $M_{\alpha \cdot n}$ has $1$ row, to give a scalar in $\ZZ[q]$.

Tensor with the identity does not change the bounds $d_0$ and $C_0$ on degrees and coefficients. Multiplying by such matrix adds at most $d_0$ to the degree, and multiplies by at most $N^{O(k)}C_0$ the largest coefficient. We get the global bounds by multiplying the matrices together, and substituting $O(\sqrt{n})$ for $k$.
\end{proof}

We give a general algorithm to compute a one-variable, integer coefficient, polynomial invariants, using standard computer algebra techniques and the algorithm of Sections~\ref{sec:graphicalalgo}-\ref{sec:algebraic}. 

\begin{proposition}\label{thm:arithmpoly}
Let $\Cat$ be a strict ribbon category of $\ZZ[q]$-modules for the one-variable polynomial ring $\ZZ[q]$. Let $L$ be an $m$-components link with colours the free modules $V_1, \ldots, V_m$, and let $J^{\Cat}_L(V_1, \ldots , V_m) \in \ZZ[q]$ be the associated topological invariant. Assume $L$ is presented by an $n$-crossings diagram $\Diag(L)$ with carving-width $\cw$.

Assume that the dimensions of the free modules $V_1, \ldots ,V_m$ are at most $N$, and that the polynomial $J^{\Cat}_{L}(V_1, \ldots , V_m)$ has degree bounded by $d_n$ and largest coefficient in absolute value bounded by $C_n$. Then $J^{\Cat}_{L}(V_1, \ldots , V_m)$ can be computed in: 
\[
\begin{array}{c}
\displaystyle O \left( d_n \left( d_n + \log C_n \right) \cdot \Arith\left(\log(d_n\log d_n + \log C_n)\right) \times n N^{\frac{3}{2}\cw} \right.\\
			\ \ \ \ \ \ \ \ \ \ \ \ \ \ \ \ \left. + d_n \left(d_n \log d_n + \log C_n\right)^2 + d_n^2 \Arith\left(d_n \log d_n + \log C_n\right) \right)\\
\end{array}
\]
machine operations, using: 
\[
\displaystyle O\left( \log\left(d_n \log d_n + \log C_n\right) N^{\cw} + n d_n(d_n \log d_n + \log C_n) + d_n^2 \Arith(d_n \log d_n + \log C_n) \right)
\]
bits. Here, $\Arith(l) \in \widetilde{O}(l)$ is the arithmetic complexity of operations $+,-,\times,\div$ on integers encoded on at most $l$ bits, which is linear in $l$ up to a poly-logarithmic factor. 
\end{proposition}

\begin{proof}
The algorithm relies on evaluation and interpolation. For short, denote $J^{\Cat}_{L}(V_1, \ldots , V_m)$ by $P(q) \in \ZZ[q]$. 

\medskip

\noindent
{\em Evaluation.} We evaluate $P(q)$ on integer points $q \in \{0,1, \ldots, d_n\}$. Fix $q_0$ in this set, and substitute $q_0$ for $q$ in matrices $c^{\pm}_{V_i,V_j}$, $\theta^{\pm}_{V_i}$, $d_{V_i}$, and $b_{V_i}$. The algorithm of Sections~\ref{sec:graphicalalgo}-\ref{sec:algebraic} is consequently a succession of matrix multiplications, where all matrices have integer coefficients (up to some preprocessing normalisation), and the resulting $P(q_0)$ is an integer of absolute value less than 
\[
	C d_n^{d_n+1} \leq 2^{(d_n+1)\log_2 d_n + \log_2 C_n} = 2^{O(d_n \log d_n + \log C_n)}
\]

For a fixed $q_0$, we perform computation modulo the first $r$ prime numbers $2=p_1, \ldots ,p_r$ successively, such that the product $p_1 \cdots p_r$ is larger than $|P(q_0)|$. We then reconstruct $P(q_0)$ using the Chinese Remainder Theorem. The product $p_1 \cdots p_r$ is of order $2^{r \log r}$~\cite{RS62}. We take an appropriate $r$ such that $r \log r \in \Theta(d_n \log d_n + \log C_n)$, which gives $r \in O(d_n + \log C_n)$.

Reconstructing the value $P(q_0)$ from all the $(P(q_0) \mod p_i)$, $1 \leq i \leq r$, can be computed in $O(r^2 \log^2 r) = O((d_n\log d_n + \log C_n)^2)$ machine operations~\cite[Theorem 5.8]{Gathen:2003:MCA:945759}.  

Additionally, the values of all primes $p_i$, $i \leq r$, are in $O(r \log (r \log r)) = O(r \log r) = O(d_n \log d_n + \log C_n)$~\cite{RS62}. 

Denote by $\Arith(l)$ the computational complexity of performing arithmetic operations $+, -, \times$ on integers encoded on at most $l$ bits, in $\ZZ/w\ZZ$, for an integer $w \leq 2^l$. The best known estimate for $C(l)$ is: 
\[
	C(l) = O(l \log^2 (l) \  2^{O(\log^* l)}) = \widetilde{O}(l),
\] 
where $\log^*$ denotes the iterated logarithm. This describes the complexity of performing the extended Euclidean algorithm~\cite{Gathen:2003:MCA:945759} using F\"urer's method~\cite{DBLP:journals/siamcomp/Furer09}. 

\medskip

\noindent
{\em Interpolation.} We reconstruct polynomial $P(q) \in \ZZ[q]$ of degree bounded by $d_n$ using Lagrange interpolation. Lagrange interpolation gives directly a formula for $P(q)$, computable in $O(d_n^2 \Arith(d_n \log d_n + \log C_n))$ machine operations~\cite[Theorem 5.1]{Gathen:2003:MCA:945759}.

\medskip

Summing up the complexity of evaluating polynomial $P(q)$ on the first $d_n+1$ non-negative integers using the modulo reconstruction approach and running the algorithm of Sections~\ref{sec:graphicalalgo}-\ref{sec:algebraic}, and the complexity of evaluating the interpolation formula, gives the complexity of the proposition.
\end{proof}

We conclude by proving the main Theorem:

\begin{proof}{ [of main Theorem~\ref{thm:maing}] }
Fixing the category $\Cat$ and the colours $V_1, \ldots, V_m$, of dimension at most $N$, makes $N$ constant, as well as the quantities $d_0$ and $C_0$ bounding degrees and coefficients of polynomials in the matrix for braidings, twists, and (co)evaluations. It enforces $d_n = O(n)$ (the bound on degree of the output polynomial), and $C_n = 2^{O(n \sqrt{n})}$ (the bound on absolute value of coefficients of the output invariant) in the complexity analysis. Substituting values gives the result of Theorem~\ref{thm:maing}.
\end{proof}

Note that we get the following parameterized complexity result for the more general problem of quantum invariant computation, where the invariant is part of the input:

\begin{theorem}\label{thm:maingg}
The problem:

\medskip

\begin{tabular}{|l}
{\sc General quantum invariant problem:}\\
{\bf Input}: $\Cat, V_1, \ldots, V_m$, presented by braiding, twist, evaluation and co-evaluation matrices, \\
\ \ \ \ \ \ \ \ \  and $m$-components link $L$, presented by a diagram $\Diag(L)$,\\
{\bf Output}: quantum invariant $J_L^{\Cat}(V_1, \ldots, V_m)$ \\
\end{tabular}

\bigskip

\noindent
can be solved in $O(\poly(n,d_0,\log C_0) N^{\frac{3}{2} \cw})$ machine operations, where $n$ and $\cw$ are respectively the number of crossings and the carving-width of the diagram $\Diag(L)$, and $d_0$ and $C_0$ are respectively the maximal degree and maximal absolute value of coefficients of any polynomial in the input matrices.
\end{theorem}

In other words, when the polynomials in the matrices are encoded with their lists of coefficients, the input size is $\Omega(\poly(N,d_0,\log C_0) + n)$, and the general quantum invariant problem is in the parameterized complexity class {\tt XP}.

\bibliography{biblio}
\bibliographystyle{amsplain}

\end{document}